\documentclass[11pt,epsf]{article}
\usepackage{graphicx}
\usepackage{amsthm}
\usepackage{amsfonts}
\usepackage{amssymb}
\title{On the multiple zeros of a partial theta function}
\author{Vladimir Petrov Kostov\\ 
Universit\'e de Nice, 
Laboratoire de Math\'ematiques, Parc Valrose,\\ 06108 Nice Cedex 2, France,  
e-mail: kostov@math.unice.fr} 
\date{}
\newtheorem{tm}{Theorem}

\newtheorem{rem}[tm]{Remark}
\newtheorem{rems}[tm]{Remarks}
\newtheorem{lm}[tm]{Lemma}

\newtheorem{prop}[tm]{Proposition}
\newtheorem{nota}[tm]{Notation}
\begin{document} 
\maketitle 
\begin{abstract}
We consider the partial theta function 
$\theta (q,x):=\sum _{j=0}^{\infty}q^{j(j+1)/2}x^j$, 
where $x\in \mathbb{C}$ is a variable and $q\in \mathbb{C}$, $0<|q|<1$, 
is a parameter. We show that, for any fixed $q$, if $\zeta$ is a 
multiple zero of the function $\theta (q,.)$,  
then $|\zeta |\leq 8^{11}$. 
\end{abstract}
The series $\theta (q,x):=\sum _{j=0}^{\infty}q^{j(j+1)/2}x^j$ 
in the variables $q$ and $x$ 
converges for $q\in \mathbb{D}_1\backslash 0$, $x\in \mathbb{C}$, 
where $\mathbb{D}_a$ 
stands for the open disk centered at the origin and of radius $a$. 
It defines a {\em partial theta function}. (We recall that 
the Jacobi theta function is 
the sum of the series $\Theta (q,x):=\sum _{j=-\infty}^{\infty}q^{j^2}x^j$ and 
the equality $\theta (q^2,x/q)=\sum _{j=0}^{\infty}q^{j^2}x^j$ holds true.) 
For any fixed $q$, $\theta$ 
is an entire function in $x$. We regard $q$ as a parameter. 

The function $\theta$ finds applications in statistical physics 
and combinatorics (see \cite{So}), also in the theory 
of (mock) modular forms (see \cite{BrFoRh}),  
in asymptotic analysis (see \cite{BeKi}) 
and in Ramanujan-type $q$-series 
(see \cite{Wa}). Its role in the framework of a problem concerning 
 hyperbolic polynomials 
(i.e. real polynomials having all their zeros real) has been discussed in 
the papers \cite{Ha}, \cite{Pe}, \cite{Hu}, \cite{Ost}, \cite{KaLoVi}, 
\cite{KoSh} and \cite{Ko2}). This  
problem has been studied earlier by Hardy, Petrovitch and Hutchinson 
(see \cite{Ha}, \cite{Hu} and \cite{Pe}). Other facts about the function 
$\theta$ can be found in~\cite{AnBe} and~\cite{So}.

\begin{rem}\label{remrem}
{\rm It has been established in \cite{Ko5} 
that for any fixed value of the parameter $q$, the function 
$\theta$ has at most finitely-many multiple zeros. 
For $q\in (0,1)$ there exists a sequence of values of $q$, tending 
to $1$, for which 
$\theta (q,.)$ has double real negative zeros tending to $-e^{\pi}$, 
see \cite{Ko3}.}
\end{rem} 

We prove the following theorem:

\begin{tm}\label{maintm}
For any $q\in \mathbb{D}_1$, any multiple zero of $\theta$ belongs to 
the set $\overline{\mathbb{D}_{8^{11}}}$ ($8^{11}=8589934592$).
\end{tm}

\begin{lm}\label{lm02}
For any $q\in \overline{\mathbb{D}_{c_0}}$, $c_0:=0.2078750206\ldots$, 
the function $\theta$ 
has no multiple zeros.
\end{lm}

(A similar result has been formulated independently by A. Sokal and 
J. Forsg{\aa}rd.) 

\begin{proof}
Indeed, set $|x|=|q|^{-k-1/2}$, $k\in \mathbb{N}$. Then 
in the series of $\theta$ the term 
$L:=x^kq^{k(k+1)/2}$ has the largest modulus 
(equal to $|q|^{-k^2/2}$). The sum $M$ of the 
moduli of all other terms is smaller than  
$|q|^{-k^2/2}\tau (|q|)$, where $\tau :=2\sum _{\nu =1}^{\infty}|q|^{\nu ^2/2}$. 
The inequality $1\geq \tau (|q|)$ is equivalent to 
$|q|\leq c_0$. Hence for $|q|\leq c_0$ one has $|L|>M$. Moreover, 
for no zero $\zeta$ 
of $\theta$ does one have $|\zeta |=|q|^{-k-1/2}$. 
For $|q|\leq 0.108$ all zeros $\xi _k$ of $\theta$ are simple, see 
\cite{Ko4}. For any $k$ fixed and for $|q|$ close to $0$ 
one has $\xi _k\sim q^{-k}$ 
(see Proposition~10 in \cite{Ko2}). Hence for 
$|q|\leq c_0$ one has 
$|q|^{-k+1/2}<|\xi _k|<|q|^{-k-1/2}$, 
i.e. all zeros of $\theta$ are simple.
\end{proof}

\begin{proof}[Proof of Theorem~\ref{maintm}.] 
We prove the theorem first in the case $1/2\leq |q|<1$. 
We use the fact that the Jacobi theta function $\Theta$ 
has only simple zeros (see \cite{Wi}), so this is also true for the function 
$\Theta ^*(q,x)=\Theta (\sqrt{q},\sqrt{q}x)=
\sum _{j=-\infty}^{\infty}q^{j(j+1)/2}x^j$. The zeros of $\Theta ^*(q,x)$ 
are all simple and equal $\mu _s:=-1/q^s$, 
$s\in \mathbb{Z}$ (which can be deduced from the form of the 
zeros of $\Theta$, see \cite{Wi}). 
We recall that the Jacobi triple product is the 
equality  
$\Theta (q,x^2)=\prod _{m=1}^{\infty}(1-q^{2m})(1+x^2q^{2m-1})(1+x^{-2}q^{2m-1})$ 
(see \cite{Wi})
from which follows the identity
$\Theta ^*(q,x)=\prod _{m=1}^{\infty}(1-q^m)(1+xq^m)(1+q^{m-1}/x)$.  

\begin{nota}
{\rm Set $G:=\sum _{j=-\infty}^{-1}q^{j(j+1)/2}x^j$. Thus $\theta =\Theta ^*-G$. 
For given $x$ ($|x|>1$) and $q$ we denote by $\kappa$ the 
least value of $m\in \mathbb{N}$ for which one has $|xq^m|<1$. 
Set $Q:=\prod _{m=1}^{\infty}(1-q^m)$, $R:=\prod _{m=1}^{\infty}(1+q^{m-1}/x)$ and 
$U_p^s:=\prod _{m=p}^s(1+xq^m)$, $s\geq p$. By $\mathcal{C}(v,r)$, 
$v\in \mathbb{C}$, 
$r>0$, we denote the circumference (in the $x$-space) 
about $v$ and of radius $r$ and by $\mathcal{D}(v,r)$ the 
corresponding open disk. We set $X_{\rho}:=\{ x\in \mathbb{C}, |x|>\rho , 
\rho >0\}$.} 
\end{nota}

\begin{rems}\label{rems1s2}
{\rm (1) Suppose that $1-1/(n-1)\leq |q|\leq 1-1/n$, $n=3$, $4$, $\ldots$. 
For $s_1>s_2>0$ one has $|\mu _{s_1}|>1$,  
$|\mu _{s_2}|>1$ and $|\mu _{s_1}-\mu _{s_2}|>1/n$ 
(because $|1/q|\geq 1/(1-1/n)>1+1/n$). Hence the two closed disks 
$\overline{\mathcal{D}(\mu _{s_i},1/2n)}$, $i=1,2$, 
do not intersect.

(2) For $x\in X_{\rho}$, $\rho >1$, one has 
$|G|\leq \sum _{j=-\infty}^{-1}\rho ^j=1/(\rho -1)$. }
\end{rems}
 
\begin{prop}\label{propprop}
Suppose that $1-1/(n-1)\leq |q|\leq 1-1/n$, $n=3$, $4$, $\ldots$, 
and that for a given $s\in \mathbb{N}$ the circumference 
$\mathcal{C}(\mu _s,1/2n)$ (hence the closed disk $\mathcal{D}(\mu _s,1/2n)$ 
as well) 
belongs to the set $X_{8^{11}}$. Then at any point of this circumference one has 
$|\Theta ^*|>1>1/(8^{11}-1)\geq |G|$.
\end{prop}

Before proving Proposition~\ref{propprop} we deduce 
Theorem~\ref{maintm} from it. 
By the Rouch\'e theorem the functions $\Theta ^*$ and $\theta$ have one 
and the same number of zeros (counted with multiplicity) inside 
$\mathcal{C}(\mu _s,1/2n)$. 
For $\Theta ^*$ this number is $1$, hence $\theta$ has a single 
zero, a simple one,  
inside $\mathcal{C}(\mu _s,1/2n)$. For any fixed $s\in \mathbb{N}$ and for 
$|q|$ sufficiently small ($q\neq 0$) the function $\theta (q,.)$ has a 
zero $\xi _s$ close to $\mu _s$ 
(close in the sense that $(\xi _s-\mu _s)\rightarrow 0$ as 
$q\rightarrow 0$, see \cite{Ko2}). Hence this is the simple zero inside 
$\mathcal{C}(\mu _s,1/2n)$. For $0<|q|\leq 0.108$ the numbers $\xi _s$ 
are all the zeros of $\theta$ (see \cite{Ko4}); these zeros are simple. 
As $|q|$ increases, for certain values of $q$ a confluence 
of certain zeros occurs 
(see \cite{KoSh}). 

Fix $s\in \mathbb{N}$. If for $0<|q|=\alpha \leq 1-1/n$ one has 
$\mathcal{C}(\mu _s,1/2n)\subset X_{8^{11}}$, then 
this inclusion holds true for $0<|q|\leq \alpha$ as well. This means that 
for $0<|q|\leq 1-1/n$ the zeros $\xi _k$ of $\theta$ with $k\geq s$ remain 
distinct, simple and belong to the interiors of the respective 
circumferences $\mathcal{C}(\mu _k,1/2n)$. Hence for 
$1-1/(n-1)\leq |q|\leq 1-1/n$ 
and $|x|>8^{11}$ there are no multiple zeros of $\theta$. This is true for any 
$n=3,4,\ldots$. Hence for $|x|>8^{11}$ and $|q|\in [1/2,1)$, the function 
$\theta$ has no multiple zeros. 

In the proof of Proposition~\ref{propprop} we use the following lemma:

\begin{lm}\label{LLL}
Suppose that $|q|\leq 1-1/b$, $b>1$, and $|x|>1$. 
Then $|Q|\geq e^{(\pi ^2/6)(1-b)}$, $|R|\geq (1-1/|x|)e^{(\pi ^2/6)(1-b)}$ 
and $|U_{\kappa +1}^{\infty}|\geq e^{(\pi ^2/6)(1-b)}$.
\end{lm}

\begin{proof}
Indeed, $|Q|\geq S:=\prod _{m=1}^{\infty}(1-|q|^m)$. Hence 
$$\begin{array}{rcl}
\ln S&=&-\sum _{m=1}^{\infty}|q|^m-(1/2)\sum _{m=1}^{\infty}|q|^{2m}-
(1/3)\sum _{m=1}^{\infty}|q|^{3m}-\cdots\\ \\ 
&=&-|q|/(1-|q|)-|q|^2/2(1-|q|^2)-|q|^3/3(1-|q|^3)-\cdots\\ \\ 
&=&(-|q|/(1-|q|))T~~,~~{\rm where}~~ 
T~=~1+|q|/2(1+|q|)+|q|^2/3(1+|q|+|q|^2)+\cdots ~.\end{array}$$
Clearly $-|q|/(1-|q|)=1-1/(1-|q|)\in (1-b,0)$ and 
$|q|^s/(s+1)(1+|q|+\cdots +|q|^s)<1/(s+1)^2$. Hence 
$T\in (0,\sum _{s=0}^{\infty}1/(s+1)^2=\pi ^2/6=1.6449\ldots )$ and 
$|Q|\geq S\geq e^{(\pi ^2/6)(1-b)}$.

To obtain the second (resp. the third) 
inequality just observe that for $m\geq 2$ one has 
$|1+q^{m-1}/x|\geq 1-|q|^{m-1}/|x|>1-|q|^{m-1}$ 
(resp. that for $m\geq \kappa +1$ one has 
$|1+xq^m|\geq 1-|xq^{\kappa}||q^{m-\kappa}|>1-|q^{m-\kappa}|$) and then apply the 
first inequality.
\end{proof}

\begin{proof}{Proof of Proposition~\ref{propprop}.} 

\noindent (A) 
One has 
$1/4\leq (1-1/(n-1))^{n-1}\leq 1/e$ and $1/8\leq (1-1/(n-1))^n\leq 1/e$.
\vspace{1mm}

\noindent (B) Suppose that $|x|>8^{11}$. Then $\kappa >11n$. Indeed, as 
$1-1/n\geq |q|\geq 1-1/(n-1)$, 
one has $1/e^{11}\geq (1-1/n)^{11n}\geq |q|^{11n}\geq 
(1-1/(n-1))^{11n}\geq 1/8^{11}$. 
Hence $|x||q|^{11n}\geq |x|/8^{11}>1$ which implies $\kappa >11n$.
\vspace{1mm}

\noindent (C) Consider the product $U_1^{\kappa}$. For 
a point external for all open disks $\mathcal{D}(\mu _i,1/2n)$, 
$i\in \mathbb{N}$, one has 
$|xq^{\kappa}+1|=|q^{\kappa}||x-\mu _{\kappa}|\geq |q^{\kappa}|/2n$ and 
$|xq^{\kappa -1}+1|=|q^{\kappa -1}||x-\mu _{\kappa -1}|\geq |q^{\kappa -1}|/2n$. 
As $\kappa >8$ and  
$|x|\geq |q|^{1-\kappa}$, for $\nu =1$, $\ldots$, $4$ one has 
$|xq^{\nu}+1|\geq |x||q|^{\nu}(1-|xq^{\nu}|^{-1})>|x||q|^{\nu}(1-|x|^{-1/2})$,~so  
$$\begin{array}{lllll}
|xq^{\kappa}+1||xq+1||xq^2+1|&\geq&|q^{\kappa}||x|^2|q|^3
(1-|x|^{-1/2})^2/2n&\geq&  
(1-|x|^{-1/2})^2/2n\\{\rm and}&&&&\\  
|xq^{\kappa -1}+1||xq^3+1||xq^4+1|&\geq& 
|q^{\kappa -1}||x|^2|q|^7(1-|x|^{-1/2})^2/2n&\geq&  
(1-|x|^{-1/2})^2/2n~.\end{array}$$

\noindent (D) For $m\leq \kappa -2$ one has $|x||q|^m>1$ and 
$|1+xq^m|\geq |x||q|^m(1-|q|^{\kappa -1-m}/|xq^{\kappa -1})|\geq 
(1-|q|^{\kappa -1-m})$. Suppose that $l\in \mathbb{N}\cup 0$, 
$\kappa -3\geq l\geq 0$. 
By analogy with 
Lemma~\ref{LLL} one can show that $|U_{\kappa -2-l}^{\kappa -2}|\geq e^{(\pi ^2/6)(1-n)}$.
\vspace{1mm}

\noindent (E) Consider the product 
$|U_{\kappa -2-4n}^{\kappa -2}|=
\prod_{m=\kappa -2-4n}^{\kappa -2}|x||q|^m|1+q^{\kappa -1-m}/xq^{\kappa -1}|$. The 
largest of the factors $|x||q|^m|$ is obtained for $m=\kappa -2-4n$. 
It equals 
$|x||q|^{\kappa -2}|q|^{-4n}>|q|^{-4n}>e^4>8^{\pi ^2/6}+1$ (because 
$8^{\pi ^2/6}=30.5\ldots <e^{3.5}=33.1\ldots$). Thus for 
$m<\kappa -2-4n$ the inequalities $|1+xq^m|\geq |xq^m|-1>8^{\pi ^2/6}$ hold true. 
\vspace{1mm}

\noindent (F) To prove the proposition it remains to show that 
$|\Theta ^*|=|Q||R||U_1^{\infty}|>1$. Set $U_{1}^{\infty}=
U_{\kappa +1}^{\infty}(1+xq^{\kappa})(1+xq^{\kappa -1})
U_{\kappa -2-4n}^{\kappa -2}U_{5}^{\kappa -3-4n}U_1^4$. 
It was shown in (C) that 
$|(1+xq^{\kappa})(1+xq^{\kappa -1})U_1^4|\geq (1-|x|^{-1/2})^4/4n^2~(*)$.
As $\kappa >11n$, there are at least $5n$ factors in the product 
$U_{5}^{\kappa -3-4n}$, and by (E) their moduli are $>8^{\pi ^2/6}$. Denote 
by $P_1$ and $P_2$ the products respectively of $4n$ and $n$ of these factors 
(assumed all distinct). Using Lemma~\ref{LLL} one finds that  
$|P_1||Q||R||U_{\kappa +1}^{\infty}||U_{\kappa -2-4n}^{\kappa -2}|
\geq (8^{\pi ^2/6})^{4n}(1-|x|^{-1})e^{4(\pi ^2/6)(1-n)}>1~(**)$ 
(because $e<8$ and $(1-|x|^{-1})e^{4(\pi ^2/6)}>1$) and 
$|P_2|(1-|x|^{-1/2})^4/4n^2>1~(***)$. Thus 
Proposition~\ref{propprop} follows from 
inequalities $(*)$, $(**)$ and $(***)$.
\end{proof} 

Now we prove Theorem~\ref{maintm} for $c_0\leq |q|\leq 1/2$.  
Lemma~\ref{LLL} implies that 
for $c_0\leq |q|\leq 1/2$ and $|x|>8^{11}$ one has 
$|Q|\geq c_1:=S|_{|q|=1/2}=0.2887880950$, 
$|R|\geq (1-|x|^{-1})c_1>0.2887880949=:c_2$ and $|U_{\kappa +1}^{\infty}|\geq c_1$.
Indeed, $|Q|\geq S$ and $S$ is minimal for $|q|=1/2$. 

We need to modify the proof of Proposition~\ref{propprop} so that it should 
become valid also for $c_0\leq |q|<1/2$. We observe first that 
$\kappa \geq 15$, with equality for $|x|=8^{11}$, $|q|=c_0$. 
Instead of the disks 
$\mathcal{D}(\mu _i,1/2n)$ we consider the disks 
$\mathcal{D}(\mu _i,1/4)$; their respective radii  
are defined 
by the conditions $|q|\leq 1-1/n$ and $|q|\leq 1-1/2$, see part (1) 
of Remarks~\ref{rems1s2}. Thus the displayed inequalities of part (C) 
of the proof of the proposition and inequality $(*)$ of part (F) remain valid 
with $n$ replaced by $2$. 

Set $U_{1}^{\infty}=
U_{\kappa +1}^{\infty}(1+xq^{\kappa})(1+xq^{\kappa -1})
U_{5}^{\kappa -2}U_1^4$. The factor $U_{5}^{\kappa -2}$ contains at least $9$ 
factors and their respective moduli are not less than $8^{11}c_0^s-1$, 
$s=5$, $\ldots$, $13$. Thus

$$\begin{array}{lclcl}|\Theta ^*|&\geq&|Q||R||U_{\kappa +1}^{\infty}|
|(1+xq^{\kappa})(1+xq^{\kappa -1})U_1^4||U_{5}^{\kappa -2}|&&\\ \\ 
&\geq&c_1c_2c_1(1/16)(1-8^{-5.5})^4\prod_{s=5}^{13}(8^{11}c_0^s-1)
&>&1~.~~~~~\Box 
\end{array}$$ 
\end{proof}

\begin{rem}
{\rm The number $8^{11}$ in the formulation of the theorem 
seems not to be optimal. 
The optimal number is not less than $e^{\pi}$, see Remark~\ref{remrem}.}
\end{rem}

\end{document}